\documentclass[11pt,a4paper,twoside,english]{article}

\usepackage{amssymb,amsmath,amsthm}
\usepackage{graphicx}
\usepackage{babel}
\usepackage{color}
\usepackage[round]{natbib}

\def\bs{\boldsymbol}

\def\argmax{\mathop{\rm arg\,max}}

\def\DTV{\mathop{d_{\rm TV}}}

\def\btp{b_{t\bs{p}}}

\def\Qtp{Q_{t\bs{p}}}

\def\rtp{r_{t\bs{p}}}

\def\Wtp{W_{t\bs{p}}}

\def\JJ{\mathcal{J}}

\def\Poiss{\mathop{\mathrm{Poiss}}\nolimits}

\def\Ex{\mathop{\mathrm{I\!E}}\nolimits}
\def\Pr{\mathop{\mathrm{I\!P}}\nolimits}

\def\Var{\mathop{\mathrm{Var}}\nolimits}

\newcommand{\R}{\mathbb{R}}

\newtheoremstyle{localthm}
	{5pt} 
	{5pt} 
	{\sl} 
	{} 
	{\bf} 
	{{\rm.}} 
	{.7em} 
	{} 

\theoremstyle{localthm}
\newtheorem{Theorem}{Theorem}
\newtheorem{Proposition}{Proposition}

\newtheoremstyle{localrem}
	{5pt} 
	{5pt} 
	{\rm} 
	{} 
	{\bf} 
	{{\rm.}} 
	{.7em} 
	{} 

\theoremstyle{localrem}

\newtheorem{Remark}[Theorem]{Remark}
\begin{document}

\addtolength{\baselineskip}{0.35\baselineskip}

\title{\bf The Density Ratio of Poisson Binomial versus Poisson Distributions}
\author{Lutz D\"{u}mbgen (University of Bern){\footnote{Research supported by Swiss National Science Foundation}} \\
	\ and\\
	Jon A.\ Wellner (University of Washington, Seattle){\footnote{Research supported in part by:
	(a) NSF Grant DMS-1566514; and 
	(b) NI-AID Grant 2R01 AI291968-04}}
}
\date{\today}
\maketitle

\begin{abstract}
Let $b(x)$ be the probability that a sum of independent Bernoulli random variables with parameters $p_1, p_2, p_3, \ldots \in [0,1)$ equals $x$, where $\lambda := p_1 + p_2 + p_3 + \cdots$ is finite. We prove two inequalities for the maximum of the density ratio $b(x)/\pi_\lambda(x)$, where $\pi_\lambda$ is the probability mass function of the Poisson distribution with parameter $\lambda$.
\end{abstract}


\section{Introduction and main results}

We consider independent Bernoulli random variables 
$Z_1, Z_2, Z_3, \ldots \in \{0,1\}$ with parameters 
$\Pr(Z_i = 1) = \Ex(Z_i) = p_i \in [0,1)$ and their sum 
$X = \sum_{i\ge 1} Z_i$. By the first and second Borel--Cantelli lemmas, 
$X$ is almost surely finite if and only if the sequence $\bs{p} = (p_i)_{i \ge 1}$ satisfies
\begin{equation}
\label{eq}
	\lambda := \sum_{k=1}^\infty p_k \ < \ \infty ,
\end{equation}
and we exclude the trivial case $\lambda = 0$. Under this assumption, the distribution $Q = Q_{\bs{p}}$ of $X$ is given by
\begin{equation}
\label{eq:b(x)}
	b(x) = b_{\bs{p}}(x) := \Pr(X = x)
	\ = \ \sum_{J \in \JJ(x)} \, \prod_{i \in J} p_i \prod_{k \in J^c} (1 - p_k)
\end{equation}
for integers $x \ge 0$, where $\JJ(x) := \{J \subset \mathbb{N} : \#J = x\}$ and $J^c := \mathbb{N} \setminus J$.

It is well-known that the distribution $Q$ may be approximated by the Poisson distribution 
$\Poiss_\lambda$ with probability mass function $\pi = \pi_\lambda$ given by 
$\pi(x) = e^{-\lambda} \lambda^x / x!$, provided that the quantity
\[
	\Delta \ := \ \lambda^{-1} \sum_{i\ge 1} p_i^2
\]
is small. Indeed, \cite{Barbour_Hall_1984} obtained the remarkable bound
\[
	\DTV(Q, \Poiss_\lambda) \ \le \ (1 - e^{-\lambda}) \Delta
\]
via a suitable version of Stein's method developed by \cite{Chen_1975}. Here $\DTV(\cdot,\cdot)$ stands for total variation distance. Note also that $\Var(X) = \sum_{i \ge 1} p_i(1 - p_i) = \lambda (1 - \Delta)$, and
\[
	\Delta \ \le \ p_* \ := \ \max_{i\ge 1} p_i .
\]

\paragraph{Main results.}
Motivated by \cite{Duembgen_Samworth_Wellner_2020}, we are aiming at upper bounds for the maximal density ratio
\[
	\rho(Q,\Poiss_\lambda) \ := \ \sup_{x \ge 0} \, r(x)
\]
with $r(x) = r_{\bs{p}}(x) := b(x)/\pi(x)$. Note that the probability mass functions $b$ and $\pi$ are densities (in the sense of the Radon-Nikodym theorem) of $Q$ and $\Poiss_\lambda$ with respect to counting measure on the set $\mathbb{N}_0$ of nonnegative integers. Thus $r = b/\pi_\lambda$ is the ``density ratio'' in the title. For arbitrary sets $A \subset \mathbb{N}_0$, the probability $Q(A) = \Pr(X \in A)$ is never larger than the corresponding Poisson probability times $\rho \bigl( Q,\Poiss_\lambda \bigr)$, no matter how small the Poisson probability is. Hence, $\rho(Q,\Poiss_\lambda)$ is a strong measure of error when $Q$ is approximated by $\Poiss_\lambda$, see also Remark~\ref{rem:TV} below. While \cite{Duembgen_Samworth_Wellner_2020} obtained explicit and essentially sharp bounds for $\rho(Q,P)$ for various pairs of distributions $P$ and $Q$, the present setting with the particular Poisson binomial distribution $Q$ and $P = \mathrm{Poiss}_\lambda$ seems to be substantially more difficult. In this note we prove the following result:

\begin{Theorem}
\label{thm:BinPoiss.pmax}
For any sequence $\bs{p}$ of probabilities $p_i \in [0,1)$ with $\lambda = \sum_{i \ge 1} p_i < \infty$,
\[
	\rho(Q, \Poiss_\lambda)
	\ \le \ (1 - p_*)^{-1} .
\]
\end{Theorem}

We conjecture that Theorem~\ref{thm:BinPoiss.pmax} is true with $\Delta$ in place of $p_*$. In the case of $\lambda \le 1$ we can prove the following result:

\begin{Theorem}
\label{thm:BinPoiss.Delta}
For any sequence $\bs{p}$ of probabilities $p_i \in [0,1)$ with $\lambda = \sum_{i \ge 1} p_i \le 1$,
\[
	\Delta \Bigl( 1 - \frac{\Delta}{2} - \frac{\lambda}{2(1 - p_*)} \Bigr)
	\ \le \ \log \rho(Q, \Poiss_\lambda)
	\ \le \ \Delta .
\]
\end{Theorem}

In particular, $\lambda \le 1$ implies that $\rho(Q, \Poiss_\lambda) \le e^\Delta < 1/(1 - \Delta)$. And since $\Delta \le p_* \le \lambda$, Theorem~\ref{thm:BinPoiss.Delta} implies that
\[
	\frac{\log \rho(Q,\Poiss_\lambda)}{\Delta} \ \to \ 1
	\quad\text{as} \ \lambda \to 0 .
\]

\begin{Remark}[Total variation distance]
\label{rem:TV}
Proposition~1~(a) of \cite{Duembgen_Samworth_Wellner_2020} implies that $\DTV(Q,\Poiss_\lambda) \le Q(\{b > \pi\}) \bigl( 1 - \rho(Q,\Poiss_\lambda)^{-1} \bigr)$. Since $b(0) = \prod_{i \ge 1} (1 - p_i)$ satisfies the two inequalities $1 - \lambda \le b(0) < e^{-\lambda} = \pi(0)$, we obtain the inequality $Q(\{b > \pi\}) \le 1 - b(0) \le \min(1,\lambda)$ and the bounds
\begin{align*}
	\DTV(Q,\Poiss_\lambda) \
	&\le \ \min(1,\lambda)
		\bigl( 1 - \rho(Q,\Poiss_\lambda)^{-1} \bigr) \\
	&\le \ \begin{cases}
		\min(1,\lambda) p_* \\
		\displaystyle
		\lambda (1 - e^{-\Delta}) \ \le \ \lambda\Delta
			= \sum\nolimits_{i\ge 1} p_i^2
			& \text{if} \ \lambda \le 1 .
		\end{cases}
\end{align*}
\end{Remark}

The remainder of this note is structured as follows: In Section~\ref{sec:Preparations} we provide some basic formulae for the probability masses $b(x)$ and the ratios $r(x)$. Then we present the proofs of Theorems~\ref{thm:BinPoiss.pmax} and \ref{thm:BinPoiss.Delta} in Section~\ref{sec:Proofs}.

\section{Auxiliary results}
\label{sec:Preparations}

\subsection{The probability mass function of $Q$}

Since $b(0) < 1$ (see Remark~\ref{rem:TV}), we know that $\rho(Q,\Poiss_\lambda) = \sup_{x \ge 1} \, r(x)$. Writing
\[
	\prod_{i \in J} p_i \prod_{k \in J^c} (1 - p_k)
	\ = \ \prod_{i \in J} \frac{p_i}{1 - p_i} \prod_{k \ge 1} (1 - p_k)
	\ = \ b(0) \ \prod_{i \in J} \frac{p_i}{1 - p_i} ,
\]
equation~\eqref{eq:b(x)} may be reformulated as
\[
	b(x) \ = \ b(0) \sum_{J \in \JJ(x)} W(J)
\]
with
\[
	W(J) \ := \ \prod_{i\in J} q_i
	\quad\text{and}\quad
	q_i \ := \ \frac{p_i}{1 - p_i} \ \in \ [0,\infty) ,
\]
i.e.\ $p_i = q_i/(1 + q_i)$. Note also that the support of $Q$ is equal to an integer interval containing $0$. Precisely,
\[
	b(x) \ > \ 0
	\quad\text{if and only if}\quad
	x \ \le \ \#\{i \ge 1 : p_i > 0\} \in \mathbb{N} \cup \{\infty\} .
\]

\subsection{Discrete scores}

For any $x \ge 0$,
\[
	\frac{\pi(x+1)}{\pi(x)} \ = \ \frac{\lambda}{x+1} ,
\]
so the ``scores'' $r(x+1)/r(x)$ are given by
\[
	\frac{r(x+1)}{r(x)} \ = \ \frac{(x+1) b(x+1)}{\lambda b(x)}
\]
for $x \ge 0$ with $b(x) > 0$. If $x_o$ is a maximizer of $r(\cdot)$, then
\begin{equation}
\label{ineq:argmax}
	\frac{(x_o+1) b(x_o+1)}{b(x_o)} \ \le \ \lambda \ \le \ \frac{x_o b(x_o)}{b(x_o-1)}
\end{equation}
with $b(-1) := 0$.

There are various ways to represent the ratios $b(x+1)/b(x)$. The following notation will be useful for that task: For any set $J \subset \mathbb{N}$, we define
\[
	s(J) \ := \ \sum_{i \in J} p_i
	\quad\text{and}\quad
	S(J) \ := \ \sum_{i \in J} q_i .
\]
In case of $x := \# J < \infty$ we set
\[
	\bar{s}(J) \ := \ s(J) / x ,
	\quad
	\bar{S}(J) \ := \ S(J) / x
	\quad\text{and}\quad
	\bar{W}(J) \ := \ W(J) \Big/ \sum_{L \in \JJ(x)} W(L)
\]
with the convention $0/0 := 0$. The numbers $\bar{W}(J)$ are probability weights in the sense that $\sum_{J \in \JJ(x)} \bar{W}(J) = 1$ whenever $b(x) > 0$. In that case,
\begin{align*}
	\frac{b(x+1)}{b(0)} = \sum_{L \in \JJ(x+1)} W(L) \
	&= \ \sum_{L \in \JJ(x+1)} \frac{1}{x+1} \sum_{k \in L} W(L \setminus \{k\}) q_k \\
	&= \ \frac{1}{x+1} \sum_{J \in \JJ(x)} W(J) \sum_{k \in J^c} q_k \\
	&= \ \frac{1}{x+1} \sum_{J \in \JJ(x)} W(J) S(J^c) .
\end{align*}
Consequently,
\begin{equation}
\label{eq:B.ratio.1}
	\frac{(x+1)b(x+1)}{b(x)}
	\ = \ \sum_{J \in \JJ(x)} \bar{W}(J) S(J^c) .
\end{equation}
Alternatively, if $b(x+1) > 0$, then
\begin{align*}
	\frac{b(x)}{b(0)} = \sum_{J \in \JJ(x)} W(J) \
	&= \ \sum_{J \in \JJ(x)} W(J)
		\sum_{k \in J^c} \frac{q_k}{S(J^c)} \\
	&= \ \sum_{J \in \JJ(x)} \sum_{k \in J^c}
		\frac{W(J \cup \{k\})}{q_k + S((J \cup \{k\})^c)} \\
	&= \ \sum_{L \in \JJ(x+1)} W(L) \sum_{k \in L} \frac{1}{q_k + S(L^c)} .
\end{align*}
Consequently,
\begin{equation}
\label{eq:B.ratio.2a}
	\frac{b(x)}{(x+1) b(x+1)}
	\ = \ \sum_{L \in \JJ(x+1)} \bar{W}(L)
		\frac{1}{x+1} \sum_{k \in L} \frac{1}{q_k + S(L^c)} .
\end{equation}
One can repeat the previous arguments with the sums $\sum_{k \in J^c} p_j/s(J^c) = 1$ in place of $\sum_{k \in J^c} q_k/S(J^c) = 1$. This leads to
\[
	\frac{b(x)}{b(0)}
	\ = \ \sum_{J \in \JJ(x)} \sum_{k \in J^c}
		\frac{W(J) p_k}{p_k + s((J \cup \{k\})^c)}
	\ = \ \sum_{L \in \JJ(x+1)} W(L) \sum_{k \in L} \frac{1 - p_k}{p_k + s(L^c)} ,
\]
because $W(J) p_k = W(J \cup \{k\}) (1 - p_k)$ for $k \in J^c$. Consequently,
\begin{equation}
\label{eq:B.ratio.2b}
	\frac{b(x)}{(x+1)b(x+1)}
	\ = \ \sum_{L \in \JJ(x+1)} \bar{W}(L)
		\frac{1}{x+1} \sum_{k \in L} \frac{1 - p_k}{p_k + s(L^c)} .
\end{equation}

Analyzing equation~\eqref{eq:B.ratio.2b} leads to a first result about the location of maximizers of $r(\cdot)$:

\begin{Proposition}
\label{prop:argmax}
Any maximizer $x_o \in \mathbb{N}_0$ of $r(\cdot)$ satisfies the inequalities $1 \le x_o \le \lceil\lambda\rceil$.
\end{Proposition}

\begin{proof}[\bf Proof of Proposition~\ref{prop:argmax}]
The inequality $x_o \ge 1$ follows from $r(0) < 1$, see Remark~\ref{rem:TV}. To verify the inequality $x_o \le \lceil\lambda\rceil$, it suffices to show that $r(x+1)/r(x) < 1$ for any integer $x \ge \lambda$ with $b(x) > 0$. This is equivalent to
\begin{equation}
\label{ineq:B.ratio.2b}
	\frac{b(x)}{(x+1)b(x+1)} \ > \ \lambda^{-1} .
\end{equation}
If $b(x+1) = 0$, this inequality is trivial. Otherwise, the left hand side of \eqref{ineq:B.ratio.2b} is given by \eqref{eq:B.ratio.2b}. Since $(1 - y)/(y + s(L^c))$ is a strictly convex function of $y \ge 0$, Jensen's inequality implies that
\[
	\frac{1}{x+1} \sum_{k \in L} \frac{1 - p_k}{p_k + s(L^c)}
	\ > \ \frac{1 - \bar{s}(L)}{\bar{s}(L) + s(L^c)}
	\ = \ \frac{1 - \bar{s}(L)}{\bar{s}(L) + \lambda - s(L)}
	\ = \ \frac{1 - \bar{s}(L)}{\lambda - x \bar{s}(L)} .
\]
But in case of $x \ge \lambda$,
\[
	\frac{1 - \bar{s}(L)}{\lambda - x \bar{s}(L)}
	\ \ge \ \frac{1 - \bar{s}(L)}{\lambda - \lambda \bar{s}(L)} \ = \ \lambda^{-1} ,
\]
whence \eqref{ineq:B.ratio.2b} holds true.
\end{proof}

Finally, let us mention that the probability mass function $b$ is ultra-log-concave in the sense that $\log r = \log(b/\pi)$ is concave, i.e.\ $r(x+1)/r(x)$ is monotone decreasing in $x \in \{y \ge 0 : b(y) > 0\}$, see Section~4 of \cite{Saumard_Wellner_2014} and the references therein. Equivalently, $(x+1) b(x+1)/b(x)$ is monotone decreasing in $x \in \{y \ge 0 : b(y) > 0\}$. With a direct argument one can even show a stronger result.

\begin{Proposition}
\label{prop:ultra.log.concave}
The ratio $(x+1) b(x+1)/b(x)$ is strictly decreasing in $x \in \{y \ge 0: b(y) > 0\}$.
\end{Proposition}

\begin{proof}[\bf Proof of Proposition~\ref{prop:ultra.log.concave}]
We have to show that for any integer $x \ge 0$ with $b(x+1) > 0$,
\[
	\frac{(x+2)b(x+2)}{b(x+1)} \ < \ \frac{(x+1) b(x+1)}{b(x)} .
\]
It follows from \eqref{eq:B.ratio.1} that the left hand side equals $S(\mathbb{N}) - \sum_{L \in \JJ(x+1)} \bar{W}(L) S(L)$ while the right hand side equals $S(\mathbb{N}) - \sum_{J \in \JJ(x)} \bar{W}(J) S(J)$. Thus the assertion is equivalent to
\begin{equation}
\label{ineq:ultra.log.concave}
	\sum_{J \in \JJ(x), L \in \JJ(x+1)} W(J) W(L) \bigl( S(L) - S(J) \bigr)
	\ > \ 0 .
\end{equation}
But each pair $(J,L) \in \JJ(x) \times \JJ(x+1)$ is uniquely determined by the three sets $M := J\cap L$, $K := (J\setminus M) \cup (L\setminus M)$ and $L' := L \setminus M$, and
\[
	W(J) W(L) \ = \ W(M)^2 W(K)
	\quad\text{and}\quad
	S(L) - S(J) \ = \ 2 S(L') - S(K) .
\]
Moreover, $\# K = 2x+1 - 2\#M$ and $\#L' = x+1 - \#M$. Hence, the left hand side of \eqref{ineq:ultra.log.concave} equals
\begin{equation}
\label{eq:ultra.log.concave.3}
	\sum_{s=0}^x \sum_{M \in \JJ(s)} \sum_{K \in \JJ(2x+1-2s)}
		1_{[M\cap K = \emptyset]}
				W(M)^2 W(K) H(K)
\end{equation}
with
\begin{align*}
	H(K) \
	:= \ &\sum_{L' \subset K \,:\, \#L' = x+1-s}
			\bigl( 2 S(L') - S(K) \bigr) \\
	 = \ &\sum_{i \in K} q_i \sum_{L' \subset K \,:\, \#L' = x+1-s}
	 		(2 \cdot 1_{L'}(i) - 1) \\
	 = \ &S(K) \binom{2x-2s}{x-s} \big/ (x+1-s) .
\end{align*}
Hence, all summands in \eqref{eq:ultra.log.concave.3} are non-negative, and $W(M)^2 W(K) S(K) > 0$ for suitable sets $M \in \JJ(x)$ and $K \in \JJ(1)$ with $M \cap K = \emptyset$.
\end{proof}

\subsection{Log-density ratios along a ray}

In what follows we consider the sequence $t\bs{p}$ for arbitrary $t \in (0,1]$, leading to the distributions $\Qtp$ with probability mass functions $\btp$, weights $\Wtp(J)$ and sums $S_{t\bs{p}}(J)$. The corresponding Poisson probability mass functions are $\pi_{t\lambda}$, and this leads to the ratios $\rtp$. According to Proposition~\ref{prop:argmax},
\[
	f(t) \ := \ \log \rho(\Qtp, \Poiss_{t\lambda})
	\ = \ \max_{1 \le x \le \lceil t\lambda\rceil} \, \log \rtp(x)
	\ = \ \max_{1 \le x \le \lceil \lambda\rceil} \, \log \rtp(x) .
\]

Now we analyze the functions $L_x : (0,1] \to \R$,
\begin{align*}
	L_x(t) \
	:=& \ \log r_{t\bs{p}}(x) \\
	 =& \ t\lambda + \log \Bigl( (t\lambda)^{-x} x! \sum_{J\in\JJ(x)}
	 	\prod_{i\in J} \frac{tp_i}{1 - tp_i} \prod_{k\ge 1} (1 - tp_k) \Bigr) \\
	 =& \ t\lambda + \sum_{k\ge 1} \log(1 - tp_k)
	 	+ \log \Bigl( \lambda^{-x} x! \sum_{J\in\JJ(x)}
			\prod_{i\in J} \frac{p_i}{1 - tp_i} \Bigr) ,
\end{align*}
for integers $x \ge 0$ with $b(x) > 0$. Note first that $L_x(t)$ can be extended to a real-analytic function of $t \in (-\infty,1/p_*) \supset [0,1]$, and
\begin{align*}
	L_x(0) \
	&= \ \log \Bigl( \lambda^{-x} x! \sum_{J \in \JJ(x)}
		\prod_{i\in J} p_i \Bigr) \\
	&\le \ \log \Bigl( \lambda^{-x} \sum_{i(1),\ldots,i(x) \ge 1}
		\prod_{s=1}^x p_{i(s)} \Bigr)
		\ = \ \log(\lambda^{-x} \lambda^x) \ = \ 0
\end{align*}
with equality for $x = 0,1$ and strict inequality for $x > 1$. This shows already that $f$ is a Lipschitz-continuous function on $(0,1]$ with limit $f(0\,+) = 0$.

Concerning the first derivative of $L_x$, for $t \in (0,1]$,
\[
	\frac{d}{dt} \prod_{i\in J} \frac{p_i}{1 - tp_i}
	\ = \ \sum_{k\in J} \frac{p_k^2}{(1 - tp_k)^2}
		\prod_{i\in J\setminus\{k\}} \frac{p_i}{1 - tp_i}
	\ = \ \prod_{i\in J} \frac{p_i}{1 - tp_i} \sum_{k\in J} \frac{p_k}{1 - tp_k} ,
\]
whence
\begin{align*}
	L_x'(t) \
	&= \ \lambda - \sum_{k\ge 0} \frac{p_k}{1 - tp_k}
		+ \sum_{J \in \JJ(x)} \, \prod_{i\in J} \frac{p_i}{1 - tp_i}
			\sum_{k \in J} \frac{p_k}{1 - tp_k}
			\Big/ \sum_{J \in \JJ(x)} \, \prod_{i\in J} \frac{p_i}{1 - tp_i} \\
	&= \ \lambda - \frac{1}{t} \Bigl( S_{t\bs{p}}(\mathbb{N})
		- \sum_{J\in \JJ(x)} \bar{W}_{t\bs{p}}(J) S_{t\bs{p}}(J) \Bigr) \\
	&= \ \lambda - \frac{1}{t} \sum_{J\in \JJ(x)} \bar{W}_{t\bs{p}}(J) S_{t\bs{p}}(J^c) .
\end{align*}
Combining this formula with \eqref{eq:B.ratio.1} yields
\begin{align}
\label{eq:DLx}
	L_x'(t) \
	&= \ \lambda - \frac{1}{t} \frac{(x+1) \btp(x+1)}{\btp(x)} \\
\nonumber
	&= \ \lambda - \lambda \frac{\rtp(x+1)}{\rtp(x)} \\
\nonumber
	&= \ \lambda \bigl( 1 - \exp \bigl( L_{x+1}(t) - L_x(t) \bigr) \bigr) .
\end{align}
In particular,
\begin{equation}
\label{ineq:LxLx+1}
	L_x'(t) \ \left\{\!\!\begin{array}{c}
		> \\[-0.5ex] = \\[-0.5ex] <
	\end{array}\!\!\right\} \ 0
	\quad\text{if and only if}\quad
	L_x(t) \ \left\{\!\!\begin{array}{c}
		> \\[-0.5ex] = \\[-0.5ex] <
	\end{array}\!\!\right\} \ L_{x+1}(t) .
\end{equation}

There is also an explicit expression for the second derivative of $L_x$: If $b(x+1) = 0$, then $x = n = \#\{i \ge 1 : p_i > 0\}$ and $L_x(t) = \lambda t + \log(\lambda^{-n} n! \, b(n))$, whence $L_x'' \equiv 0$. Otherwise, for $0 < t \le 1$,
\[
	L_x''(t) \ = \ \lambda \exp \bigl( L_{x+1}(t) - L_x(t) \bigr)
		\bigl( L_x'(t) - L_{x+1}'(t) \bigr) ,
\]
and
\[
	L_x'(t) - L_{x+1}'(t)
	\ = \ \frac{1}{t} \Bigl( \frac{(x+2)\btp(x+2)}{\btp(x+1)}
		- \frac{(x+1)\btp(x+1)}{\btp(x)} \Bigr)
	\ < \ 0
\]
by Proposition~\ref{prop:ultra.log.concave}. Hence $L_x$ defines a smooth concave function on $[0,1]$.

\section{Proofs of the main results}
\label{sec:Proofs}

\begin{proof}[\bf Proof of Theorem~\ref{thm:BinPoiss.pmax}]
We know that $f(t) = \log \rho(\Qtp,\Poiss_{t\lambda})$ is equal to the maximum of $L_x(t)$ over $x \in \{1,\ldots,\lceil\lambda\rceil\}$, and that $f(0\,+) = 0$. Note also that
\[
	f'(t\,+) \ = \ \max_{x \in N(t)} L_x'(t)
\]
where
\[
	N(t) \ := \ \argmax_{x \in \{1,\ldots,\lceil\lambda\rceil\}} \, \rtp(x) .
\]
Since $g(t) := - \log(1 - tp_*)$ satisfies $g(0) = 0$ and $g'(t) = p_*/(1 - tp_*)$, it suffices to show that
\[
	L_x'(t) \ \le \ \frac{p_*}{1 - tp_*}
	\quad\text{for any} \ x \in N(t) .
\]
According to \eqref{eq:DLx}, the latter requirement is equivalent to
\[
	\frac{(x+1)\btp(x+1)}{\btp(x)} \ \ge \ t\lambda - \frac{tp_*}{1 - tp_*}
	\quad\text{for any} \ x \in N(t) .
\]
Note that $x \in N(t)$ implies that $L_{x-1}(t) \le L_x(t)$. But the latter inequality is equivalent to $L_{x-1}'(t) \le 0$, see \eqref{ineq:LxLx+1}, and by \eqref{eq:DLx}, this is equivalent to
\[
	\frac{x \btp(x)}{\btp(x-1)} \ \ge \ t\lambda .
\]
Consequently, it suffices to show that
\[
	\frac{(x+1)\btp(x+1)}{\btp(x)} \ \ge \ t\lambda - \frac{tp_*}{1 - tp_*}
	\quad\text{whenever}\quad
	\frac{x \btp(x)}{\btp(x-1)} \ \ge \ t\lambda .
\]
We may simplify notation by replacing $t\bs{p}$ with $\bs{p}$ and prove that
\begin{equation}
\label{ineq:fancy}
	\frac{(x+1)b(x+1)}{b(x)} \ \ge \ \lambda - \frac{p_*}{1 - p_*}
	\quad\text{whenever}\quad
	\frac{x b(x)}{b(x-1)} \ \ge \ \lambda .
\end{equation}
Note that for $1 \le x \le \lceil\lambda\rceil$, the representation \eqref{eq:B.ratio.2a} with $x-1$ in place of $x$ reads
\[
	\frac{b(x-1)}{x b(x)}
	\ = \ \sum_{J \in \JJ(x)} \bar{W}(J)
		\frac{1}{x} \sum_{i \in J} \frac{1}{q_i + S(J^c)} .
\]
By Jensen's inequality,
\[
	\frac{1}{x} \sum_{i \in J} \frac{1}{q_i + S(J^c)}
	\ \ge \ \Bigl( \frac{1}{x} \sum_{i \in J} (q_i + S(J^c)) \Bigr)^{-1}
	\ = \ \bigl( \bar{S}(J) + S(J^c) \bigr)^{-1} ,
\]
so
\[
	\frac{b(x-1)}{x b(x)}
	\ \ge \ \sum_{J \in \JJ(x)} \bar{W}(J) \bigl( \bar{S}(J) + S(J^c) \bigr)^{-1} .
\]
A second application of Jensen's inequality yields that
\[
	\frac{b(x-1)}{x b(x)}
	\ \ge \ \Bigl( \sum_{J \in \JJ(x)}
		\bar{W}(J) \bigl( \bar{S}(J) + S(J^c) \bigr) \Bigr)^{-1} .
\]
Consequently, if $x b(x) / b(x-1) \ge \lambda$, then
\[
	\sum_{J \in \JJ(x)} \bar{W}(J) \bigl( \bar{S}(J) + S(J^c) \bigr)
	\ \ge \ \lambda .
\]
On the other hand, \eqref{eq:B.ratio.1} yields
\begin{align*}
	\frac{(x+1)b(x+1)}{b(x)} \
	&= \ \sum_{J \in \JJ(x)} \bar{W}(J) \bigl( \bar{S}(J) + S(J^c) \bigr)
		- \sum_{J \in \JJ(x)} \bar{W}(J) \bar{S}(J) \\
	&\ge \ \lambda - \frac{p_*}{1 - p_*} ,
\end{align*}
because $\bar{S}(J) = x^{-1} \sum_{i \in J} p_i/(1 - p_i) \le p_*/(1 - p_*)$ for any set $J \in \JJ(x)$. This proves \eqref{ineq:fancy}.
\end{proof}

\begin{proof}[\bf Proof of Theorem~\ref{thm:BinPoiss.Delta}]
We know from Proposition~\ref{prop:argmax} that in case of $\lambda \le 1$,
\[
	\log \rho(Q, \Poiss_\lambda)
		\ = \ \log r(1) \ = \ L_1(1)
\]
with
\[
	L_1(t) \ = \ t\lambda + \sum_{i\ge 1} \log(1 - tp_i)
		+ \log \Bigl( \lambda^{-1} \sum_{i\ge 1} \frac{p_i}{1 - t p_i} \Bigr) .
\]
First of all, $L_1(0) = 0$, and
\begin{align*}
	L_1'(t) \
	&= \ \lambda - \sum_{i\ge 1} \frac{p_i}{1 - tp_i}
		+ \sum_{i\ge 1} \frac{p_i^2}{(1 - tp_i)^2}
			\Big/ \sum_{i\ge 1} \frac{p_i}{1 - tp_i} \\
	&= \ - t \sum_{i \ge 1} \frac{p_i^2}{1 - tp_i}
		+ \sum_{i\ge 1} \frac{p_i^2}{(1 - tp_i)^2}
			\Big/ \sum_{i\ge 1} \frac{p_i}{1 - tp_i} ,
\end{align*}
whence $L_1'(0) = \Delta$. Moreover, we have seen before that $L_1'' \le 0$ by ultra-log-concavity of the probability mass functions $\btp$. Consequently, for some $\xi \in (0,1)$,
\[
	L_1(1)
	\ = \ L_1(0) + L_1'(0) + 2^{-1} L_1''(\xi)
	\ = \ 0 + \Delta + 2^{-1} L_1''(\xi)
	\ \le \ \Delta .
\]

As to the lower bound, recall that
\[
	L_1(1)
	\ = \ \sum_{i\ge 1} (p_i + \log(1 - p_i))
		+ \log \Bigl( \lambda^{-1} \sum_{i\ge 1} \frac{p_i}{1 - p_i} \Bigr) .
\]
On the one hand,
\[
	p_i + \log(1 - p_i)
	\ = \ - \sum_{k\ge 2} \frac{p_i^k}{k}
	\ \ge \ - \frac{p_i^2}{2} \sum_{\ell \ge 0} p_*^\ell
	\ = \ - \frac{p_i^2}{2(1 - p_*)} ,
\]
so
\[
	\sum_{i\ge 1} (p_i + \log(1 - p_i))
	\ \ge \ - \frac{1}{2(1 - p_*)} \sum_{i\ge 1} p_i^2
	\ = \ - \frac{\lambda}{2(1 - p_*)} \Delta .
\]
Moreover,
\[
	\log \Bigl( \lambda^{-1} \sum_{i\ge 1} \frac{p_i}{1 - p_i} \Bigr)
	\ \ge \ \log \Bigl( \lambda^{-1} \sum_{i\ge 1} (p_i + p_i^2) \Bigr)
	\ = \ \log(1 + \Delta)
	\ \ge \ \Delta - \Delta^2/2 , 
\]
and this implies the asserted lower bound for $L_1(1)$.
\end{proof}

\paragraph{Acknowledgement.}
Part of this research was conducted at the Mathematical Research Institute Oberwolfach (MFO), Germany, in June and July 2019. We are grateful to the MFO for its generous hospitality and support. We also thank a referee and an associate editor for constructive comments.


\end{document}